\documentclass[12pt]{article}
\usepackage{geometry}                % See geometry.pdf to learn the layout options. There are lots.
\usepackage{graphicx}
\usepackage[mathscr]{eucal}
\usepackage{amsfonts}
\usepackage{amsmath,amssymb,amscd,amsthm}
\usepackage{subfigure}
\usepackage{graphicx,color}
\usepackage[flushleft,alwaysadjust]{paralist}

\newtheorem{theorem}{Theorem}
\newtheorem{definition}{Definition}
\newtheorem{lemma}{Lemma}

\newtheorem{proposition}{Proposition}
\theoremstyle{definition}
\newtheorem{remark}{Remark}

\input{epsf}
\newcommand{\beq}{\begin{equation}}

\newcommand{\eeq}{\end{equation}}
\newcommand{\R}{{\mathbb R}}
\newcommand{\pd}{\partial}
\def \c#1{$^{\ref{#1}}$}
\def \r#1{\refstepcounter{prova}\label{#1}
$^{\ref{#1}}$} 
\newcounter{prova}
\title{Gravitational and Harmonic Oscillator Potentials on Surfaces of Revolution}
\date{}
\begin{document}
\maketitle
\begin{center}
\author{Manuele Santoprete}
\end{center}
\begin{center}
{\footnotesize Department of Mathematics\\
Wilfrid Laurier University\\
75 University Avenue West,\\
Waterloo, ON, Canada, N2L 3C5.\\
msantopr@wlu.ca\\
}\end{center}
%\date{}                                           % Activate to display a given date or no date
\begin{abstract}
In this paper we consider the motion of a particle on a surface of revolution
under the influence of a central force field. 
We prove that there are at most two analytic central potentials for which all the bounded, non-singular  orbits are closed and that  there are exactly two on some surfaces with constant Gaussian curvature. The two potentials leading to closed orbits are suitable generalizations of the gravitational and harmonic oscillator potential. We also show that there could be  surfaces admitting only one potential that leads to closed orbits.  In this case the potential is  a generalized harmonic oscillator.   
In the special case of surfaces of revolution with constant Gaussian curvature we prove a generalization of the well-known Bertrand Theorem.
\end{abstract}
\vskip 0.5truecm
\hspace{25pt}  PACS(2006): 45.05.+x, 45.50.Pk

%%%%%%%%%%%%%%%%%%%%%%%%
\section*{\large\bf I. Introduction}
%%%%%%%%%%%%%%%%%%%%%%%%%

The problem of describing the  motion of a particle on surfaces of constant curvature, under the influence of a central potential, is an interesting  problem that dates back to  the 19th century. 
 Lobachevski\c{Lobachevski} was probably the first to propose an analogue of the gravitational force of Newton for the Hyperbolic space $H^3$.
In 1860 Serret\c{Serret} generalized the gravitational force to the sphere and solved the Kepler problem on $S^2$.
In 1870 Schering\c{Schering} wrote an analytical expression for the Newtonian potential on $H^3$.
Only three years later Lipschitz\c{Lipschitz} considered a one body motion in a central potential on the two-sphere $S^2$.
In 1885 Killing\c{Killing} found a generalization of all three Kepler's laws to the case of a three-sphere $S^3$. 

The extension of these results to the hyperbolic case was carried out by Liebmann in 1902.\c{Liebmann1902} He also derived generalizations of the oscillator potential
for $S^3$ and $H^3$. 

The well known Bertrand theorem (that states that there are only two analytic central potentials in Euclidean space for which all the bounded orbits are closed) was generalized to the spaces $S^2$ and $H^2$ by Liebmann in 1903.\c{Liebmann}

Many of these classical results have been long forgotten (see Ref. \ref{Shchepetilov} for more details).
However since then many  authors have studied the classical Kepler problem and the quantum analogue (the hydrogen atom) rediscovering the old results and introducing new elegant ones (see Ref. \ref{Carinena} for some interesting results and for an extensive bibliography on the subject). 

New interest on the topic was generated, at least in part,  because of cosmological models as the 
mixmaster universe\c{Misner} where the spatial slices are positively curved and are topologically three-spheres $S^3$.

In this paper we study the  motion of a particle on surfaces of revolution, under the influence of a central potential. This is a generalization of the analogous problem on surfaces of constant curvature.

We first  generalize the  gravitational potential to surfaces of revolution in two different ways. The first method is  viewing the gravitational potential as a solution of the Laplace-Beltrami equation. The second one is a generalization of an approach of Appell\c{Appell} (see also Ref. \ref{Albouy} and \ref{Borisov}). In this case we define the gravitational potential and the harmonic oscillator potential on surfaces of revolution relating them to the planar case. 

 The potential of the gravitational interaction and the harmonic oscillator on the plane have a peculiar property: they are the only  potentials that have the Bertrand property, i.e. that  generate a central field where all the bounded non-singular  orbits are closed. 
Note, however, that there are non-potential forces all of whose bounded orbits are closed. See for example Ref. \ref{Whittaker} pages 79, 80.
 %Note that  in this context we use a non-standard definition of bounded orbits. For instance consider a circular crown on the plane, according to our definition,  the orbits that reach the boundary are  ``unbounded". See Section VI for a precise definition.
 It is therefore natural to ask weather or not the gravitational potential on surfaces of revolution lead to closed orbits. We show that, in general this is not the case. Indeed such potentials lead to bounded orbits only on certain surfaces of revolution with constant Gaussian curvature.  

Another of the main results of the paper is the proof of Bertrand's theorem for surfaces of revolutions with constant Gaussian curvature: we show that  on certain surfaces the only potentials for which all the bounded non-singular orbits are closed are the  generalization of the gravitational and the harmonic potential.   This result generalizes the proof of Liebmann\c{Liebmann}  (that holds in the case of the sphere $S^2$ and the hyperbolic plane $H^2$) and Kozlov and Harin\c{Kozlov} (that holds in the case of the sphere). 
Note that, while in the case of the Euclidean plane the two-sphere and the hyperbolic plane all the bounded orbits close after one ``loop", this is not true in general for surfaces of revolution with constant Gaussian curvature. Indeed in the latter case a non-circular orbit will close after $n$ loops (where $n$ is an integer that depends on the surface).   

%Note that while  a surface of constant curvature is locally isometric to either the Euclidean plane or the two-sphere or the hyperbolic plane proving Bertrand's theorem on those three surfaces doesn't address the problem in the general case

Finally we prove that, for a general surface of revolution, there are at most two central potentials that lead to bounded closed orbits and there are exactly two on some surfaces of constant Gaussian curvature, in which case  the potentials are the generalization of the gravitational and the harmonic ones. It is worth noticing that on certain surfaces (e.g. the torus) there are no potentials leading to closed orbits. 
We also show that there could be surfaces of revolution where there is only one potential leading to closed orbits and such potential is the generalized harmonic oscillator.  
We were unable to find any explicit example of this last kind of surfaces. 

The proofs use a suitable generalization of a proof of the  classical Bertrand's theorem due to Tikochinsky.\c{Tikochinsky} However, we were unable to obtain a proof based on Arnol'd's treatment of Bertrand's theorem (see Ref. \ref{Arnold}, Section 2.8D).  The basic idea is the following: first we treat  circular orbits of radius $u_0$. These are shown to exist for potentials defined on the surfaces under consideration. Next we derive a condition for closed orbits. Then, we consider  small deviations from $u_0$ and, using the condition above, we expand the effective potential to the first non vanishing order. This leads to a first condition that is expressed in the form of a 
differential equation. Finally, we  use the next two orders in the expansion of the 
effective potential and find a further condition for closed orbits. The two conditions are then analyzed and used to obtain the main results in the paper.  

The paper is organized as follows. In the next section  we write the equations of motion of a central potential on a surface of revolution. In Sec. III, we define the gravitational potential and the harmonic oscillator potential on a surface of revolution. 
In Sec. IV we find an expression for the Gaussian curvature of a surface of revolution and we prove several facts important in the case of constant curvature. 
In Sec. V, we
% use a method of A.C. Clairaut (1713-1765) to 
write the equations of the trajectory on a surface of revolution and we show that the gravitational potential and the harmonic potential lead to closed orbits for certain surfaces of constant Gaussian curvature. 
In the last section we state and prove the main results of the paper. 
%%%%%%%%%%%%%%%%%%%%%%%%
\section*{\large\bf II. Equations of motion}
%%%%%%%%%%%%%%%%%%%%%%%%%
Let $I$ be an interval of real numbers then we say that  $\gamma:I\rightarrow \R^2$ is a regular plane curve if $\gamma$ is $C^1$  and
$\gamma'(x)\neq 0$ for any $x\in I$.
% The curve $\gamma$ is said to be %simple if it injective in the interior of $I$. 
%It is said to be closed if $I$ is a closed bounded interval $[a,b]$ and $\gamma(a)=\gamma(b)$.

\begin{definition}Let $\gamma: I\rightarrow \R^3$ be a simple (no self intersections) regular plane curve $\gamma(u)=((f(u),0,g(u))$ on the $xz$-plane where $f$ and $g$ are smooth curves on the interval $I$, with $f(u)>0$ in the interior of $I$. 
Let $S$ be a surface isometrically  embedded in $\R^3$ that admits a parametrization ${\bf x}:I\times \R\rightarrow S$ of the form
\beq{\bf x}(u,\phi)=(f(u)\cos\phi,f(u)\sin\phi,g(u))
\label{eqparam}\eeq
then:
\begin{compactenum}
\setlength{\itemsep}{0.1cm} 
\item if $I=[c,d]$ and $f(c)=f(d)=0$, $S$ is a {\it spherical surface of revolution}
\item if $I=(c,d)$, with $-\infty\leq c<d\leq \infty$, $S$ is a {\it hyperboloidal surface of revolution}
\item if $I=[c,d]$ and $\gamma(c)=\gamma(d)$ with $f(c)=f(d)>0$,  $\gamma$ is a closed loop and $S$ is a {\it toroidal surface of revolution}.
\item if $I=[c,d)$, with $\infty <c<d\leq \infty$ and $f(c)=0$, then $S$ is a {\it paraboloidal surface of revolution}
\end{compactenum}
In all cases $S$ is a surface of revolution obtained by rotating $\gamma$ about the $z-$axis. The curve $\gamma$ will be called the profile curve.

\end{definition}
Note that a spherical surface of revolution is isomorphic to $S^2$ and that by definition the sets ${\bf x }(c,\phi)$ and ${\bf x}(d,\phi)$ reduce to  single points, i.e. the {\it north} and the {\it south poles} of S. Similarly hyperboloidal, toroidal and paraboloidal surfaces of revolution are homeomorphic to a hyperboloid of one sheet, a torus ($S^1\times S^1$) and an elliptic  paraboloid, respectively. 
Metric singularities can occur only on spherical and paraboloidal surfaces of revolution.
If $S$ is a spherical surface of revolution metric singularities can only occur at the north and south poles, S is smooth everywhere else.
If $S$ is a paraboloidal surface of revolution metric singularities can occur only at $u=c$.
Hyperboloidal and toroidal surfaces of revolutions do not have metric singularities and are smooth.
 
Throughout this paper all surfaces of revolution will be assumed to be  as in Definition 1 (i.e. they will be either spherical, hyperboloidal, toroidal or paraboloidal) and the profile curve $\gamma$ is assumed to be unit speed, i.e. $ (\frac{df}{du})^2+ (\frac{dg}{du})^2=1$.

%%%%%%%%%%%%%%
\iffalse  
Let $S \subset{\mathbb R}^3$ be the set obtained by rotating a simple   regular plane curve $\gamma$ around a straight line in the plane. For instance take the $xz$-plane as the plane of the curve and the $z$ axis as the rotation axis. Let
\[\gamma(u)=((f(u),0,g(u)), \quad u\in I\]
with $f(u)>0$ in the interior of $I$, $f(a)\geq 0$ and $f(b)\geq 0$ be   a parametrization of the curve $\gamma$. Denote by $\phi$ the rotation angle about the $z$ axis. Thus we obtain a map
\beq{\bf x}(u,\phi)=(f(u)\cos\phi,f(u)\sin\phi,g(u))
\label{eqparam}\eeq
from the %open
set $U=\{(u,\phi)\in{\mathbb R}^2; u\in I, \phi\in[0,2\pi) \}$ into $S$. $S$ is % a regular surface that is
called a {\it surface of revolution}.  It is regular if $f(u)>0$ for $u\in I$ or if $f(u)$ intersects the $z$ axis horizontally.
\fi
%%%%%%%%%%%%%%%%%
For a surface of revolution $S$, a simple computation gives the coefficients of the first fundamental form, or metric tensor (subscripts denote partial derivatives): 
\[E={\bf x}_u \cdot {\bf x}_u=\left(\frac{df}{du}\right)^2+ \left(\frac{dg}{du}\right)^2=1, \quad F={\bf x}_u \cdot {\bf x}_\phi=0 \quad G={\bf x}_\phi \cdot {\bf x}_\phi=f(u)^2,\]
 so that the metric (away from any singular point) is 
\beq
ds^2= E~du^2+2F~du~d\phi+G~d\phi^2=  du^2+f(u)^2d\phi^2.
\label{eqds}\eeq
Note that the parametrization is orthogonal ($F=0$) and that $E_\phi=G_\phi=0$.
Surfaces given by parametrizations with these properties are said to be {\it $u$-Clairaut}.
The Lagrangian function of a particle of mass $m$ moving on the surface takes the form
\[
L=\frac m 2 ( \dot u^2+f(u)^2\dot \phi^2) -V(u,\phi)
\]
where $V(u,\phi)$ is the potential energy.  We now consider the case where $V$ is a function of $u$ alone, i.e. it is a {\it central potential}. Furthermore we assume that $V$ is analytic except, at most, at the points where $f(u)=0$, where the function is allowed to have a singularity. 

The Hamiltonian is 
\[
H=\frac{p_u^2}{2m}+\frac{p_\phi^2}{2m f(u)^2}+V(u)
\]
where $p_\phi=mf(u)^2\dot\phi$.

{\bf Examples:}
Motion on the plane: take $f(u)=u$, $g(u)=0$ with $u\in (0,\infty)$.
In this case one recovers the usual central force problem.    

Motion on the sphere: take $f(u)=\sin(u)$, $g(u)=\cos(u)$ with $u\in [0,\pi]$.

Equations of motion:
\[\left\{\begin{array}{l}
\dot u=\frac{\partial H}{\partial p_u}=\frac{p_u}{m}\\
\dot\phi=\frac{\partial H}{\partial p_\phi}=\frac{p_\phi}{m f(u)^2}\\
\dot p_u=-\frac{\partial H}{\partial u}=\frac{p_\phi^2f'(u)}{mf(u)^3}    -\frac{dV}{du} \\
\dot p_\phi=-\frac{\partial H}{\partial \phi}=0
\end{array}\right.
\]
Clearly $H$ and $p_\phi$ are constant of motions, they are in involution and the problem is integrable by  the Liouville-Arnold theorem. 

Since $V:(c,d)\rightarrow \mathbb{R}$ is real analytic, standard results of differential equation theory guarantee, for any initial data $(u(0),\phi(0),p_u(0),p_\phi(0))$ the existence and uniqueness of an analytic solution defined on a maximal interval $[0,t^*)$, where $0<t^*\leq \infty$. If  if $t^*<\infty$, we say the solution  is {\it singular}.
If the potential is singular at $u=c$ and/or $u=d$ this singularity induces singularities in the solution. If $u(t)\rightarrow {c}$ and/or $u(t)\rightarrow {d}$  as $t \rightarrow t^*$ we say that the solution experience a {\it collision}.  
It can be shown that, in the problem under discussion, there are two types of singularities: collisions, and the singularities that arise when a solution reach the boundary of the surface of revolution in a finite time. 

%%%%%%%%%%%%%%%%%%%%%%%%%%%%%%%%%%%%%%%
\section*{\large\bf III. Gravitational and Harmonic potential for surfaces of revolution}
%%%%%%%%%%%%%%%%%%%%%%%%%%%%%%%%%%%%%%%%
In this section we generalize the gravitational and the harmonic oscillator potential to general surfaces of revolution. We present two different ways to do so. 
The first one starts from the observation that the gravitational potential is a solution of the Laplace equation. It is then natural to define the gravitational potential on a surface of revolution as a solution of the Laplace-Beltrami equation. 
The second is based upon the work of Appell\c{Appell} (see also Ref. \ref{Albouy} and \ref{Borisov}) that used the central  projection (or in cartographer's jargon the gnomonic projection) to relate the motion on the plane to the motion on a sphere.
%%%%%%%%%%%%%%%%%%%%%%%%%%%%%%%%%
\subsection*{\large\bf A. Laplace-Beltrami Equation}
%%%%%%%%%%%%%%%%%%%%%%%%%%%%%%%%%%%

The Laplace-Beltrami equation generalizes the Laplace equation to arbitrary surfaces.
For a function $V$ depending only on $u$, if the element of length is given by equation 
(\ref{eqds}), the Laplace-Beltrami equation takes the form 
\beq
\triangle V(u)=\frac{1}{f(u)^2}\frac{\partial}{\partial u}\left(f(u)^2\frac{\partial V(u)}{\partial u}\right)=0
\label{eqlaplace}\eeq

The solution of the Laplace-Beltrami equation  is 
\beq
V_1(u)=a \Theta(u)
\label{eqV}
\eeq
where $a$ is a  constant and $\Theta(u)$ is an antiderivative of $1/f(u)^2$. 
To be more definite let us assume $a>0$. The parameter $a$ plays the role of the gravitational constant. This generalizes the gravitational potential to surfaces of revolution. The analogue of the harmonic oscillator potential instead  is given by
%\[
%V_2(u)=k\left (\int \frac{du}{f(u)^2}\right)^{-2}.
%\]

\beq
V_2(u)=k \Theta(u)^{-2}.
\label{eqoscillator}
\eeq

%%%%%%%%%%%%%%%%%%%
\subsection*{\large \bf B. Central Projection}
%%%%%%%%%%%%%%%%%%%
Following Serret  \c{Serret} Appell\c{Appell} consider a system in $\R^2$ with the following equations of motion (in polar coordinates):
\beq
\frac{d}{d\tau}\left(\frac{\pd T_p}{\pd(dr/d\tau)}\right)=R, \quad \frac{d}{d\tau}\left(\frac{\pd T_p}{\pd (d\psi/d\tau)}\right)=\Psi\label{eqplane}
\eeq
where $T_p$ is the kinetic energy of a point mass (of mass $m=1$) in the plane 
\[
T_p=\frac 1 2 \left(\left(\frac{dr}{d\tau}\right)^2+r^2\left(\frac{d\psi}{d\tau}\right)^2\right)
\]
while $R,\Psi$ stand for certain generalized forces.

Similarly let $T_s$ be the kinetic energy of a point mass on the surface of revolution $S$.
\[
T_s=\frac 1 2 ( \dot u^2+f(u)^2 \dot\phi^2).
\]
 The equations of motion are
\beq
\frac{d}{d\tau}\left(\frac{\pd T_s}{\pd \dot u}\right)={\mathcal U}, \quad \frac{d}{d\tau}\left(\frac{\pd T_s}{\pd \dot\phi}\right)=\Phi \label{eqsurface}
\eeq
Consider the transformation of coordinates and time given by
\beq
r=X(u)=-\Theta(u)^{-1},\quad \phi=\psi,\quad d\tau=Y(u)dt=(f(u)\Theta(u))^{-2}dt.
\eeq
Then the equations (\ref{eqplane}) take the form of equations (\ref{eqsurface}) where 
\[{\mathcal U }=Y(u)R,\quad \Psi=Y(u)\Phi.\]
Now we can prove the following
\begin{theorem}
There exists a trajectory isomorphism  between the Lagrangian system on $\R^2$ with central potential 
\beq
L_p=\frac 1 2 \left(\left(\frac{dr}{d\tau}\right)^2+r^2\left(\frac{d\psi}{d\tau}\right)^2\right)+V(r)
\eeq
and the Lagrangian system on the surface of revolution $S$, given by
\beq
L_s=\frac 1 2 ( \dot u^2+f(u)^2 \dot\phi^2)+V(-\Theta(u)^{-1})
\eeq
\end{theorem}
\begin{proof}
Let $\Phi=\Psi=0$
and \[R=-\frac{\pd U}{\pd r},\quad {\mathcal U}=-\frac{\pd V}{\pd u}=-\frac{\pd  V}{\pd r}\frac{\pd r}{\pd u}=\frac{R}{Y(u)}\]
\end{proof}

In particular in the case of the Newtonian potential, then  
\[R=-\frac{a}{r^2}=-\frac{a}{X(u)}=-a\Theta(u)^2\]
and thus
\[{\mathcal U}=Y(u)R=-\frac{a}{f(u)^2}\]
integrating (and changing the sign) we find the potential 
\[
V_1=a\Theta(u)
\]
that coincides with the solution of the Laplace-Beltrami equation on the surface $S$. It is natural to consider  $V_1$ as the analogue of the gravitational potential. 
Similarly in the case of the harmonic oscillator potential
\[R=-\bar kr=-kX(u)=\bar k\Theta(u)^{-1}\]
and thus
\[
{\mathcal U}=Y(u)R=\frac{\bar k}{f(u)^2\Theta(u)^3}
\]
integrating (and changing the sign)  we find
\[
 V_2=-\bar k\int \frac{ du}{f(u)^2\Theta(u)^3}=\frac{\bar k}{2} \Theta(u)^{-2}=k\Theta(u)^{-2}
\]
where $k=\bar k/2$.
It is natural to consider  $V_2$ as the analogue of the harmonic oscillator potential.
%%%%%%%%%%%%%%%%%%%%%%%%%%%%%
\section*{\large\bf IV. Gaussian Curvature of Surfaces of Revolution}
%%%%%%%%%%%%%%%%%%%%%%%%%
Let  ${\bf x}(u,\phi)$ be a parametrization of the surface and let 
\[
{\mathcal E}=\frac{({\bf x_u}\wedge {\bf x}_\phi)\cdot {\bf x}_{uu}}{\sqrt{EG-F^2}}, \quad {\mathcal F}=\frac{({\bf x_u}\wedge {\bf x}_\phi)\cdot {\bf x}_{u\phi}}{\sqrt{EG-F^2}}, \quad {\mathcal G}=\frac{({\bf x_u}\wedge {\bf x}_\phi)\cdot {\bf x}_{\phi\phi}}{\sqrt{EG-F^2}}
\]
be the coefficients of the second fundamental form in this parametrization. 
Then the Gaussian curvature is given by the expression 
\beq
K=\frac{{\mathcal E} \mathcal G-\mathcal F^2}{EG-F^2}.
\eeq
For a surface of revolution the parametrization is given by (\ref{eqparam}) and thus $\mathcal E=-fg',~ \mathcal F=0$, and $\mathcal G=f'g''-g''f'$. Consequently the Gaussian curvature is
\[
K=-\frac{g'(g'f''-g''f')}{f}.
\]
It is convenient to put the Gaussian curvature in another form.
By differentiating $(f')^2+(g')^2=1$ we obtain $f'f''=-g'g''$. Thus,
\[
K=-\frac{g'(g'f''-g''f')}{f}=-\frac{(g')^2 f''+(f')^2f'')}{f}=-\frac{f''}{f}
\]

Now we want to study surfaces of revolution with constant curvature $K$.
The requirement of constant curvature gives us a linear differential equation to solve
\[
f''=-K f.
\]
The solutions to this differential equation are of the form
\[
f(u)=A e^{i\sqrt{K}u}+Be^{-i\sqrt{K}u}
\]
if $K\neq 0$
and 
\[
f(u)=Cu+D
\]
if $K=0$. 

Then, substituting $f(u)$ into the unit speed relation $f'(u)^2+g'(u)^2=1$ and solving for $g(u)$ gives
\beq
g(u)=\pm \int_{u_0}^u \sqrt{1-f'(s)^2}~ds
\label{eqg}
\eeq

\begin{remark}
Note that the only surfaces of revolution with zero constant curvature are the right circular cylinder, the right circular cone  and the plane. 
\end{remark}
\begin{remark}
The sphere is obtained when  $K>0$ and $A=-B=\frac{1}{2i}$.
The hyperbolic plane is obtained when $K<0$ and $A=-B=\frac 1 2$. 
\end{remark}

Now we can prove the following

\begin{proposition}
The equation 
\beq-ff''+(f')^2=b^2
\label{eqf1}\eeq 
is verified if and only if the surface of revolution $S$ has constant Gaussian curvature $K$
and either $f(u)=A e^{i\sqrt{K}u}+Be^{-i\sqrt{K}u}$ with $AB=b^2/4K$ or 
$f(u)=Cu+D$ with   $C=\pm b$.
\label{prf0}
\end{proposition}
\begin{proof}
Note that 
\beq
\left(\frac{(f')^2-b^2}{f^2}\right)'=-2\frac{ff'}{f^4}(-ff''+(f')^2-b^2)
\label{eqcurv}
\eeq

If  $-ff''+(f')^2=b^2$ then from Eq. (\ref{eqcurv}) it follows that
\[
\left(\frac{(f')^2-b^2}{f^2}\right)=-K
\]
for some constant $K$. Consequently, since $-ff''+(f')^2=b^2$, $f''/f=-K$
and the curvature is constant. 

On the other hand assume that $f''=-Kf$. Then, 
if $K\neq 0$, $f=Ae^{i\sqrt{K}u}+Be^{-i\sqrt{K}u}$. Plugging this into  
$-ff''+(f')^2=b^2$ we find the condition $AB=\frac{b^2}{4K}$.
If $K=0$ then $f=Cu+D$. Plugging into the equation we find
$C^2=b^2$.

\end{proof}

\begin{proposition}
The function $f$ satisfies the equation \beq
\frac{f'(u)}{f(u)}=-b^2\Theta(u)%\int\frac{du}{f(u)^2}
\label{eqint}
\eeq
for some antiderivative $\Theta(u)$ of $1/f(u)^2$, if and only if it satisfies the  nonlinear differential equation
\[
-ff''+(f')^2=b^2
\]\label{prf}
\end{proposition}
\begin{proof}
\[
\left(\frac{f'}{f}\right)'=\frac{f''f-(f')^{2}}{f^{2}}=-\frac{b^{2}}{f^{2}},
\]
which implies (15) for some $\Theta(u)$.

%%%%%%%%%%%%%%%%%%%%%%%%%%%
\iffalse
Assume $f$ satisfies Eq. (\ref{eqint}). Differentiating (\ref{eqint}) and multiplying by $-f^2(u)$ we obtain $-ff'' +(f')^2=b^2$.
On the other hand assume  (\ref{eqf1}) is verified then by  Proposition \ref{prf} either $f(u)=A e^{i\sqrt{K}u}+Be^{-i\sqrt{K}u}$ with $AB=1/4K$ or 
$f(u)=Cu+E$ with  and $C=\pm b$.
In the first case we have 
\[
\frac{f'(u)}{f(u)}=\frac{i\sqrt{K}(Ae^{2i\sqrt{K}u}-B)}{Ae^{2i\sqrt{K}u }+B}
\]
and 
\[
\int \frac{du}{f(u)^2}=\frac{i}{2A\sqrt{K}(Ae^{2i\sqrt{K}u}+B)}+C_1
\]
where $C_1$ is an arbitrary constant. 
Substituting the expressions above in (\ref{eqint}) we obtain the following system of equations
\[\left\{\begin{array}{l}
i\sqrt{K}A=-b^2C_1A\\
b^2\frac{2\sqrt{K}ABC_1+i}{2\sqrt{K}A}=i\sqrt{K}B.
\end{array}\right.
\]
From the first equation we obtain $C_1=-\frac{i\sqrt{K}}{b^2}$.
From the second equation we find $4KAB=b^2$, that is identically verified since $K=\frac{b^2}{4AB}$.

In the second case $f'(u)=C$
\[
\frac{C}{Cu+D}= -b^2\int \frac{du}{(Cu+D)^2}=\frac{b^2}{C(Cu+D)}+C_1
\]
that yields $C_1=0$ and $C^2=b^2$.
\fi
%%%%%%%%%%%%%%%%%%%%%%%%%%%%%
\end{proof}
%-------------------------------------------------------------------------------------
\begin{figure}[t]
  \begin{center}
    \mbox{
      \subfigure[]{\resizebox{!}{5.5cm}{\includegraphics{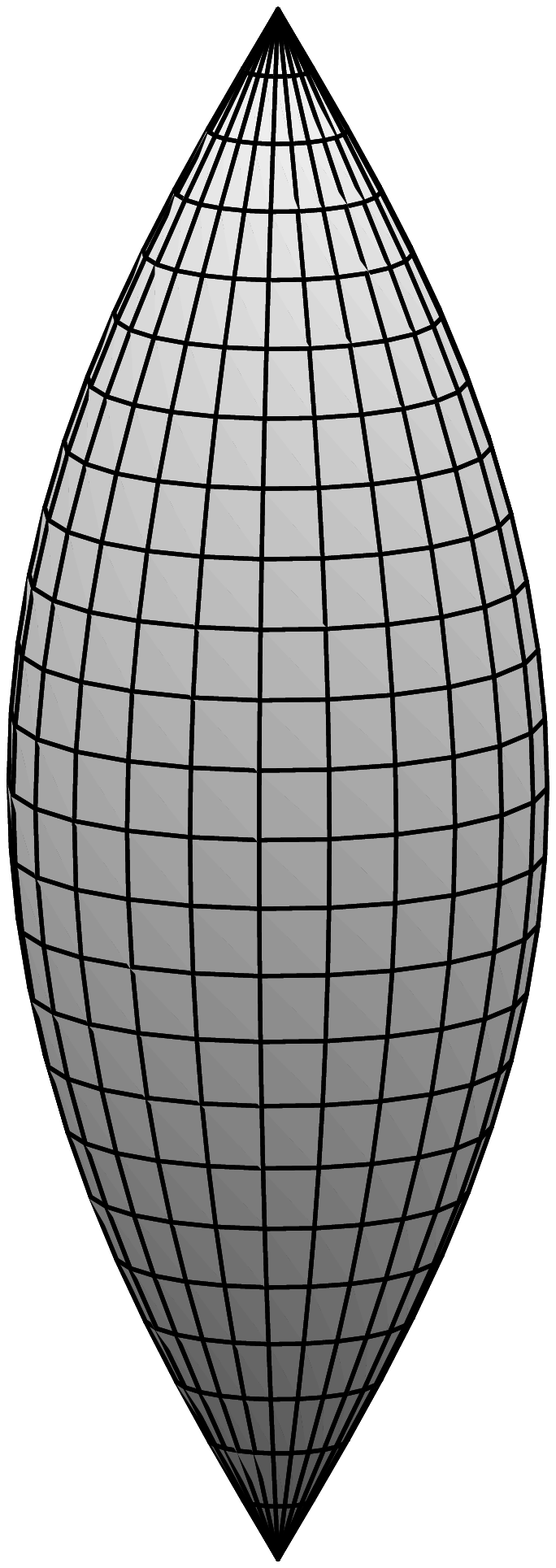}} \label{plotposconstcurv}} \quad\quad
      \subfigure[]{\resizebox{!}{5.5
      cm}{\includegraphics{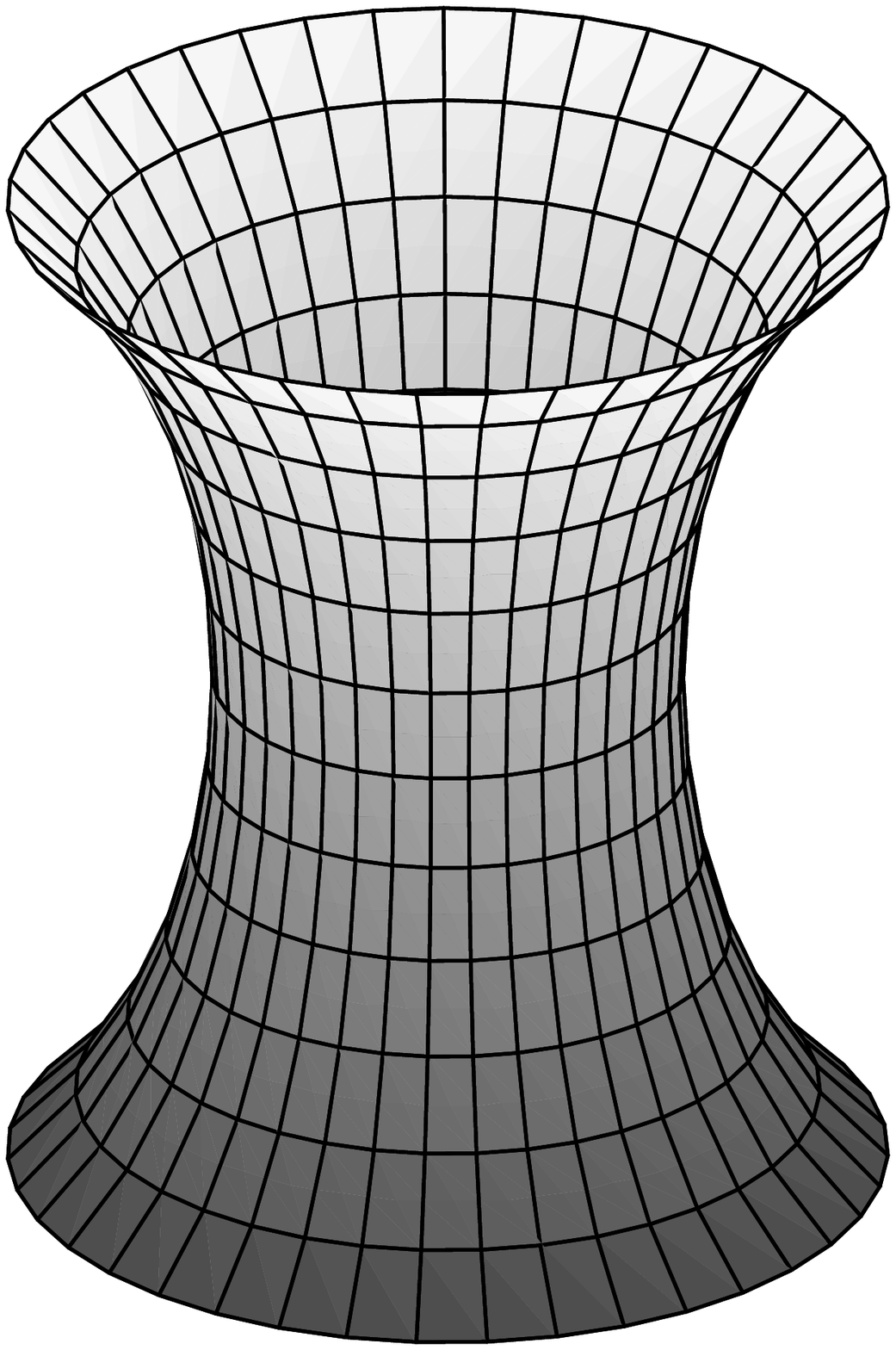}} \label{plotnegconstcurv}}} 
    \caption{(a) A constant $K=1$ surface with  $A=1/2$ and $B=1/8$ (b) A constant $K=-1$ surface with  $A=-1/2$ and $B=1/8$ }
  \end{center}
\end{figure}
%-----------------------------------------------------------------------------------------------

\begin{remark}
Note that   there are nontrivial surfaces with constant Gaussian curvature (i.e. beside the Euclidean plane the Hyperbolic plane and the sphere).
A surface  with  Gaussian curvature is $K=1$ and $b=1/2$ where  $A=1/2$ and $B=1/8$ (i.e. with  $f(u)=1/2e^{iu}+1/8e^{-iu}$) is depicted in Fig. \ref{plotposconstcurv}. In this case $f(u)$ satisfies $-ff''+(f')^2=b^2$ with $b=1/2$. 
A surface with Gaussian curvature $K=-1$, $b=1/2$, $A=1/2$ and $B=1/8$ (i.e. with $f(u)=-1/2e^{u}+1/8e^{-u}$) is given in Fig. \ref{plotnegconstcurv}. As in the previous example $f(u)$ satisfies $-ff''+(f')^2=b^2$ with $b=1/2$. 

\end{remark}

%%%%%%%%%%%%%%%%%%%%%%%%%%%%%%%%%%
\section*{\large\bf V. Equation of the Trajectory}
%%%%%%%%%%%%%%%%%%%%%%%%%%%%%%%%%%%
We now   write the equation of the trajectory.
Let $p_\phi\neq 0$. Then the coordinate $\phi$ varies monotonically and can be used as a new time. 
Let us put 
\[\rho=1/r=-\Theta(u)\]
where $\Theta(u)$ is the antiderivative of $1/f(u)^2$ selected in Proposition \ref{prf}.
This change of variable has a long and distinguished history that goes back to 
A.C. Clairaut's {\it Th\'eorie de la Lune} (1765) and it seems strictly related to  
the various proofs of Bertrand's theorem. For instance the proofs in Refs. \ref{Arnold},\ref{Goldstein} and the original proof of Bertrand 
use the change of variable above. 
%This choice is not necessary but simplifies the discussion.
 
Since $p_\phi=mf(u)^2\dot \phi$ it is clear that
\[
\dot \rho=-\frac{\dot u}{f(u)^2}, \quad \frac{d\rho}{d\phi}=-\frac{m\dot u}{p_\phi},\quad \frac{d^2\rho}{d\phi^2}=-\frac{m^2\ddot uf(u)^2}{p_\phi^2}.
\]
Consequently the equation of motion 
\[
\ddot u=\frac{p_\phi^2\frac{df}{du}}{m^2f^3(u)}-\frac 1 m \frac{dV}{du}=0
\]
can be rewritten as 
\beq
\frac{d^2\rho}{d\phi^2}+\frac{df}{du}f(u)^{-1}+\frac{ m}{p_\phi^2}\frac{dV(1/\rho)}{d\rho}=0
\label{eqrho}
\eeq
If $f(u)$ satisfies Eq. (\ref{eqf1}) then by Proposition \ref{prf} we have $\frac{df}{du}f(u)^{-1}=b^2\rho$. 
Consequently we obtain
\beq
\frac{d^2\rho}{d\phi^2}+b^2\rho+\frac{ m}{p_\phi^2}\frac{dV(1/\rho)}{d\rho}=0.
\label{eqrho1}
\eeq
This is the equation of the trajectory. In the case of the the Euclidean plane this is substantially given in Newton's {\itshape Principia}, Book I. \S\S  2 and 3. and in  A.C. Clairaut's {\it Th\'eorie de la Lune} (1765). See also Ref. \ref{Whittaker} for a more accessible reference. 

%%%%%%%%%%%%%%%%%%%%%%%%%%%%%%%%%%%%%%%%%%%%%
\iffalse
Thus the equation of the orbits has the same form as in the case of a particle moving  in a central field with potential $V(r)$.
According to Bertrand's theorem orbits will be closed in two cases:
\[
U_1(r)=-\frac{a}{r},\quad\quad U_2(r)=\frac{kr^2}{2}; \quad a,k>0.
\]
Therefore we get the following two potentials:
\[V_1=-a\rho=a\int\frac{du}{f(u)^2}, \quad V_2=\frac{k}{\rho^2}=k\left(\int\frac{du}{f(u)^2}\right)^{-2}.
\]
It is natural to consider $V_1$ as the analogous of the potential of  gravitational interaction and $V_2$ as the analogous of the potential of the harmonic oscillator. 
Note that the potential function $V_1$ coincides with the solution of the  Laplace-Beltrami Eq. (\ref{eqlaplace}). 
\fi
%%%%%%%%%%%%%%%%%%%%%%%%%%%%%%%%%%%%%%%%%%%%%%%%%%

%%%%%%%%%%%%%%%%%%%%%%%%%
\subsection*{\large\bf A. The Gravitational Case}
%%%%%%%%%%%%%%%%%%%%%%%%%
In this section we study the motion under the influence of the potential
\[
V_1=a\Theta(u)=-a\rho.
\]
In this case the equation of the trajectory (\ref{eqrho1}) takes the form \[
\frac{d^2\rho}{d\phi^2}+b^2\rho-\frac{am}{p_\phi^2}=0.
\]
The solution is given by the sum of the solution of the homogeneous equation of the form $\rho=\frac e p \cos[b(\phi-\phi_0)]$ plus a solution of the non-homogeneous equation $\rho=\frac 1 p$.
The solution $\rho=\frac 1 p$ corresponds to the ``circular orbit" of radius 
\[
p=\frac{b^2 p_\phi^2}{am}.
\] 
Consequently the trajectory is given by
\[\rho=\frac 1 p(1+e\cos[b(\phi-\phi_0)]).\]
%%%%%%%%%%%%%%%%%%%%%%%%%%%%
\subsection*{\large\bf B. The Harmonic Oscillator Case}
%%%%%%%%%%%%%%%%%%%%%%%%%%%%%%
In this section we study the motion under the influence of the potential
\[
V_2=\frac{k}{ \Theta^{2}}=\frac{k}{\rho^2}
\]
In this case the equation of the trajectory (\ref{eqrho1}) takes the form\[
\frac{d^2\rho}{d\phi^2}+b^2\rho-\frac{2km}{p_\phi^2\rho^3}=0.
\]
A first integral of the equation above is
\[
h=\frac 1 2 (\frac{d\rho}{d\phi})^2+\frac{km}{p_{\phi}^2\rho^2}+\frac 1 2b^2\rho^2.
\]
Consequently the orbital equation is
\[
\frac{d\rho}{d\phi}=\pm \sqrt{2\left(h-\frac{km}{p_\phi^2\rho^2}-\frac{b^2\rho^2}{2}\right)}
\]
and thus
\[
\phi-\phi_0=\int_{\rho_0}^{\rho(\phi)}\frac{d\rho}{\rho\sqrt{2\left(h\rho^2-\frac{km}{p_\phi^2}-\frac{b^2\rho^4}{2}\right)}}=\frac{1}{2b}\int_{\rho_0}^{\rho(\phi)}\frac{d\rho}{\rho\sqrt{-\left(\rho^2-\frac{h}{b^2}\right)^2+\left(\frac{h^2}{b^4}-\frac{2km}{p_\phi^2b^2}\right)}}.
\]
The substitution $w=\rho^2-\frac{h}{b^2}$ yields
\[\phi-\phi_0=\frac{1}{2b}\int_{w_0}^w\frac{dw}{\sqrt{-w^2+\eta^2}}\]
where $\eta^2=\left(\frac{h^2}{b^4}-\frac{2km}{p_\phi^2b^2}\right)$ and $w_0=\rho_0^2-\frac{h}{b^2}$.
Consequently, choosing $\rho_0^2=\frac{h}{b^2}$, we obtain
\[
\phi-\phi_0=-\frac{1}{2b}\arccos\left(\frac{w}{\eta}\right)
\]
and the equation of the orbit is given by
\[
\rho^2=\frac{h}{b^2}+\eta\cos[2b(\phi-\phi_0)]
\]

\begin{lemma}
All the bounded orbits given by the gravitational and harmonic oscillator potential 
on the surface of revolution $S$ are closed if $-f''+(f')^2=b^2$ where $b$ is a rational number. \label{lemmakephar}
\end{lemma}
\begin{proof}
In the  case of the gravitational potential $\rho=\frac 1 p(1+e\cos[b(\phi-\phi_0)])$ and the bounded orbits are clearly closed if $b$ is rational. 
Similarly, in the case of the harmonic oscillator,  $\rho^2=\frac{h}{b^2}+\eta\cos[2b(\phi-\phi_0)]$ and all the bounded orbits are closed provided that $b$ is a rational number.
\end{proof}

%-------------------------------------------------------------------------------------
\begin{figure}[t]
  \begin{center}
    \mbox{
      \subfigure[]{\resizebox{!}{6cm}{\includegraphics
   {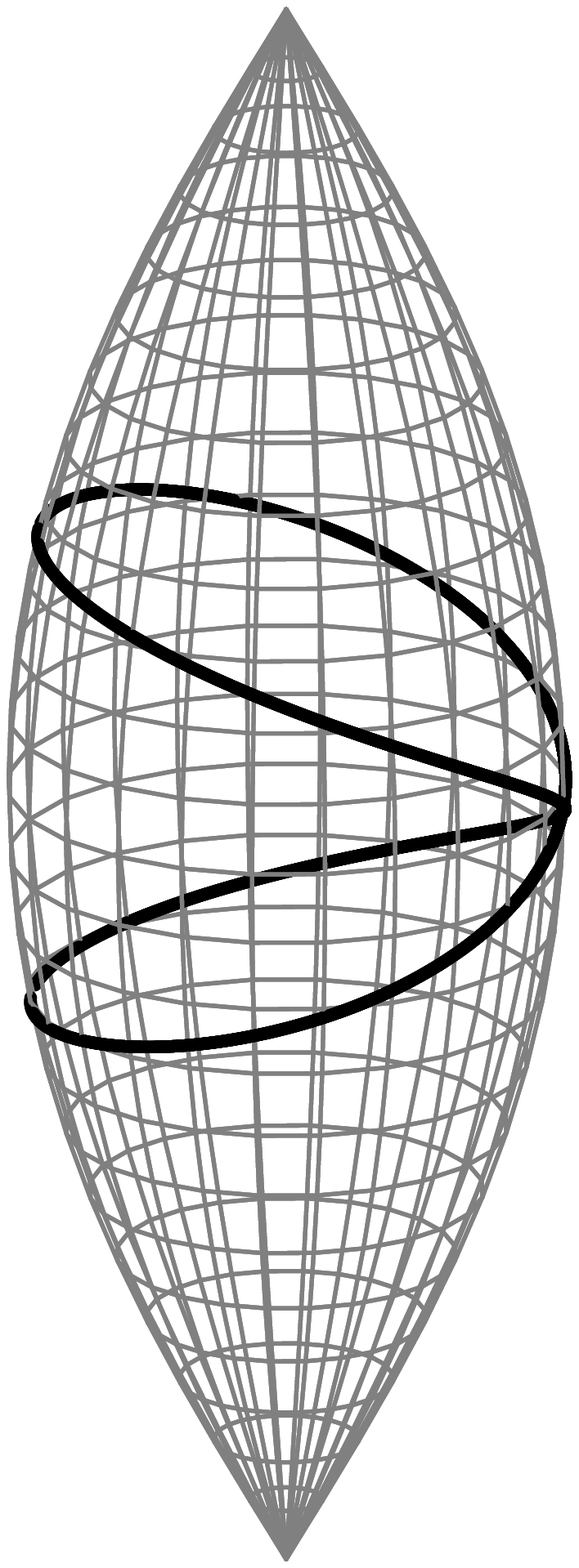}} %{posconstcurvflow.eps}}
       \label{posconstcurvflow}} \quad\quad
      \subfigure[]{\resizebox{!}{6cm}{\includegraphics{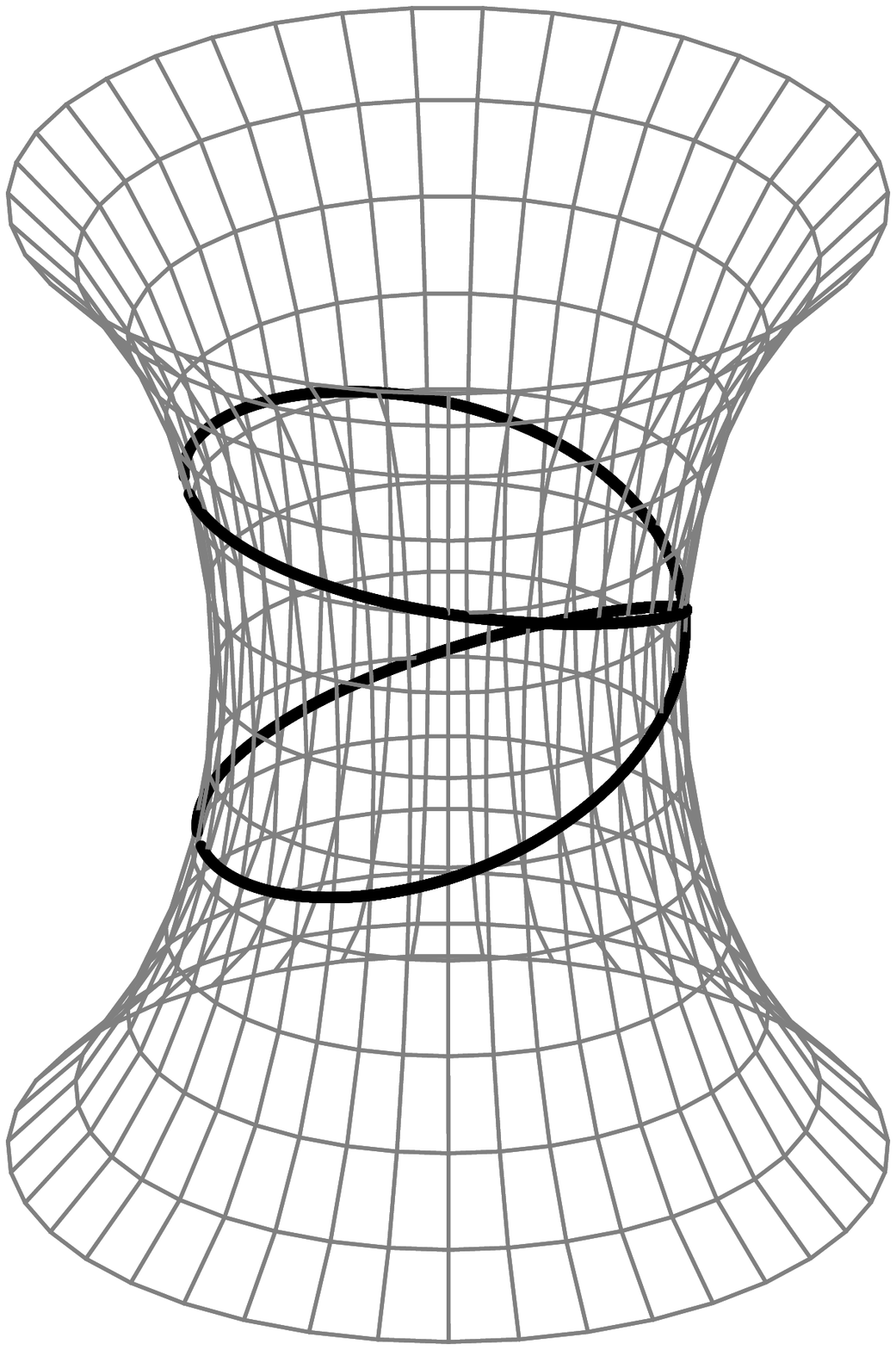}}%{negconstcurvflow.eps}} 
      \label{negconstcurvflow}}} 
    \caption{(a) A periodic orbit on a  constant $K=1$ surface with  $A=1/2$ and $B=1/8$. (b) A periodic orbit on a constant $K=-1$ surface with  $A=-1/2$ and $B=1/8$. }
  \end{center}
\end{figure}
%---------------------------------------------------------------------------------------------
A periodic orbit of the generalized gravitational potential  on a  surface with constant curvature $K=1$ (with  $A=1/2$ and $B=1/8$) is depicted in  Fig. \ref{posconstcurvflow}.  A periodic orbit of the generalized gravitational potential  on a  surface with constant curvature $K=-1$ (with  $A=1/2$ and $B=1/8$) is depicted in  Fig. \ref{negconstcurvflow}. Fig \ref{posconstcurvflow} and \ref{negconstcurvflow} depict examples of surfaces where all the orbits of the generalized potential are closed. In those examples $b=1/2$ and thus all the orbits wind around the surface twice. 

%%%%%%%%%%%%%%%
\section*{ \large \bf VI. Main Results}
%%%%%%%%%%%%%%%%
In this section we obtain the main results of the paper. In order to do that we need some definitions and several lemmas.  
Let 
\[W(u)=\frac{l^2}{2mf(u)^2}+V(u)
\]
with $l=p_\phi$ denote the effective potential. Given the energy $E$ and the angular momentum $l$ the orbit can be calculated from 
\[
\phi(u)=\phi(u_0) +\int_{u_0}^u \frac{l}{mf(u)^2}\frac{du}{\sqrt{\frac{2}{m}[E-W(u)]}}.
\]

\iffalse 
We now wan to to make precise the meaning of ``bounded orbits" in the present context. The definition differs from the standard one. 
Recall that $\R^3$ with the Euclidean distance $d$ is a metric space. 
Then the distance between two non-empty subsets $\Gamma$ and $\Theta$ of $\R^3$ is defined by
\[d(\Gamma,\Theta)=\inf\{d(x,y):x\in \Gamma,y\in\Theta\}.\] 
Let  $a$ and $b$ be the endpoints of the interval $I$.  If $\gamma(a)=\gamma(b)$  then $S$ is a compact bounded surface and all the orbits are bounded. Let  $S_a=(f(a)\cos\phi,f(a)\sin\phi,g(a))$ and $S_b=(f(b)\cos\phi,f(b)\sin\phi,g(b))$ whit $\phi\in[0,2\pi)$. We say that an orbit $\gamma$ is bounded if there exist an $\epsilon>0$ such that \[\min\{d(\gamma,S_a),d(\gamma,S_b)\}\geq\epsilon\]
This means that if, for instance, we consider an (open)  circular crown lying on the $xy$-plane the orbits that ``reach" the boundary of the crown  are unbounded according to our definition.  
\fi 
 
 To prove the main theorems  we first treat circular orbits of fixed radius $u_0$. Then we perform a first order (Lemma \ref{lemmaV1}) and a third order (Lemma \ref{lemmabeta})  study of the orbits which remain  close to the circular one $u_0$. But the existence of such orbits must first be guaranteed. In order to do that we first have to show that stable  periodic orbits exist for all the surfaces of revolutions and potentials we consider in Bertrand's theorem. 

\begin{lemma}
Consider   a central potential on a surface of revolution $S$  that has at least one bounded non-circular orbit. A necessary condition  to have all the bounded non-singular orbits closed is to have a minimum of the effective potential $W(u)$ (i.e. a stable circular orbit).
\label{lemmastart}\end{lemma}
\begin{proof}
We consider three possible cases:

\noindent
a) $S$ is a spherical   surface of revolution\\
b) $S$ is a hyperboloidal  surface of revolution \\
c) $S$ is a toroidal surface of revolution\\
d) $S$ is a paraboloidal surface of revolution\\

\noindent Case a) We distinguish several  cases. If $V$ is continuous in  $[c,d]$ it is bounded and  $W$ has a local minimum in $(c,d)$. This is because $\frac{l^2}{2mf(u)^2}\rightarrow \infty$ as $u\rightarrow c$, $u\rightarrow d$ and $V$  bounded in $[c,d]$ imply that $W(u)\rightarrow \infty$ as $u\rightarrow c$ and $u\rightarrow d$. Now consider the case  $V$ is not continuous in $u=c$ (but continuous at $u=d$). If  $W'(u)>0$  then all the orbits are collision orbits. Thus we must have $W'(u)=0$ at some point $u_0$. If it is a saddle then all the orbits are collisions. If it it is a local maximum then there must be a minimum point, since $W(u)\rightarrow \infty $ as $u\rightarrow d$. The case $V$ not continuous at $c$ is similar. Now assume  $V$ is not continuous in $u=c$ and $u=d$. If $W'(u)>0$ or $W'(u)<0$ in $(c,d)$ then all the orbits are collisions. Thus $W'(u)=0$ for some $u$ in $(c,d)$. If it is a saddle or a maximum then all the orbits are collisions. Hence it must be a minimum point. 

\vskip 0.2 cm
\noindent Case b) If $W'(u)> 0$ or $W'(u)< 0$ for every $u$ then there are no bounded non-singular orbits except, at most, the ones asymptotic to the boundary (that are not closed). Thus there must be  an $u^*\in I$ such that $W'(u^*)=0$. If at $u=u^* $ there is a local maximum or a saddle point  there are no bounded non-singular solutions besides the circular one (except at most bounded solutions asymptotic to the boundary of $S$). Thus $W(u^*)$ must be a local minimum.

\vskip 0.2 cm
\noindent Case c) The surface $S$ is compact and %$\lim_{u\rightarrow a+}W(u)=\lim_{u\rightarrow b-}W(u)$.
$W(c)=W(d)$
 Since $W(u)$ is a continuous function on $[c,d]$, differentiable  on $(c,d)$, it has a (local) maximum or a minimum at some $u^*\in(c,d)$.  
 Clearly $W'(u^*)=0$.
If $W$ has a local maximum then there are bounded orbits asymptotic to the periodic one and not all the bounded orbits are closed. Thus there must be a local minimum.  

\vskip 0.2 cm
\noindent Case d) If $W'(u)<0$ for all $u\in I$ all the solutions are either unbounded or singular. 
If $W'(u)>0$ and the potential $V(u)$ is singular at $u=c$ then all the solutions are collisions. 
On the other hand if the potential is smooth at $u=c$ then $W(u)\rightarrow \infty$ as $u\rightarrow c$, since  $\frac{l^2}{2mf(u)^2}\rightarrow \infty$ as $u\rightarrow c$. Hence if the potential is smooth $W'(u)$ cannot be positive for every $u\in I$.
Consequently there is a $u^*\in I$ such that $W'(u^*)=0$. If such point is a saddle point or a maximum  all the solutions are either unbounded or singular (except at most bounded orbits asymptotic to the boundary of $S$).   
\end{proof}

Consider a bounded motion between turning points $u_1$ and $u_2$ in the vicinity of a local minimum $u_0$  of the effective potential. Let $\Delta\phi(E)$ denote the advance in a complete journey from $u_2$ to $u_1$ and back to $u_1$  and let $W_0=W(u_0)$ be a local minimum value,   then we have the following

\begin{lemma}
Consider  a central potential on a surface of revolution $S$  and assume the effective potential $W$ has a minimum at $u_0$ and  yields closed orbits  then
\beq
\int_{u_1(W)}^{u_2(W)}\frac{ds}{f(s)^2}=\frac{2\sqrt{2m}}{l\beta}\sqrt{W-W_0}
\label{eqW}\eeq
where $\beta$ is a constant such that  $\beta=\frac{2\pi}{\Delta\phi}=\frac{p}{q}\neq 0$.
\end{lemma}
\begin{proof}
Since the orbit is symmetric about the direction of a turning point we have
\beq\begin{split}
\Delta\phi(E)&=2\int_{u_1}^{u_2} \frac{l}{mf(u)^2}\frac{du}{\sqrt{\frac{2}{m}[E-W(u)]}}\\
&=\sqrt{\frac{2}{m}}l\left[\int_E^{W_0}\frac{1}{f(u_1(W))^2}\frac{du_1(W)}{dW}\frac{dW}{\sqrt{E-W}}+\int_{W_0}^E\frac{1}{f(u_2(W))^2}\frac{du_2(W)}{dW}\frac{dW}{\sqrt{E-W}}\right]\\
&=\int_{W_0}^E \Gamma(W)\frac{dW}{\sqrt{E-W}}
\end{split}
\label{eqdeltaphi}\eeq
where
\[\begin{split}
\Gamma(W)&=\sqrt{\frac 2 m} l\left[\frac{1}{f(u_2(W))^2}\frac{du_2}{dW}-\frac{1}{f(u_1(W))^2}\frac{du_1}{dW}\right]\\
&=\sqrt{\frac 2 m} l\frac{d}{dW}\left[\int_a^{u_2(W)}\frac{ds}{f(s)^2}-\int_a^{u_1(W)}\frac{ds}{f(s)^2}\right].
\end{split}\]
Equation (\ref{eqdeltaphi}), considered as an integral equation for the unknown function $\Gamma(W)$ is a special case of Abel's equation (or Euler's hypergeometric transformation) and  can be solved for $\Gamma(W)$ in terms of $\Delta\phi(E)$ as follows.
Divide both sides by $\sqrt{\bar W-E}$ and integrate over $E$ between $W_0$ and $\bar W$
\[
\int_{W_0}^{\bar W}\frac{\Delta\phi}{\sqrt{\bar W-E}}~dE=\int_{W_0}^{\bar W}\int_{W_0}^{E}\frac{\Gamma(W)}{\sqrt{\bar W -E}\sqrt{E-W}}~dW~dE.
\]
A change in the order  of integration leads to 
\[
\int_{W_0}^{\bar W}\frac{\Delta\phi}{\sqrt{\bar W-E}}~dE=\int_{W_0}^{\bar W}\Gamma(W)dW\int_{W}^{\bar W}\frac{dE}{\sqrt{\bar W -E}\sqrt{E-W}}.
\]
The last integral is elementary. Its value is $\pi$. Let $W_0=W(u_0)$.
Since $u_1(W_0)=u_2(W_0)=u_0$ we have 
\beq
\int_{W_0}^{ W} \frac{\Delta \phi}{\sqrt{W-E}}~dE=\pi\int_{W_0}^{\bar W} \Gamma(W) dW=\pi l\sqrt{\frac 2 m} \int_{u_1(W)}^{u_2(W)} \frac{ds}{f(s)^2}.
\label{eqdelta}
\eeq
The previous equation is valid for any bounded motion. We now write it for  closed orbits.
The condition for an orbit to be closed is that $\Delta\phi(E)=q/p$ where $q$ and $p$ are integers.
If $\Delta\phi(E)/2\pi$ is a continuous function of $E$ it must be constant otherwise it would assume irrational values.
Since $\Delta\phi$, as a function of the energy is a constant, the integration in Eq. (\ref{eqdelta}) can be performed to obtain
\[
\int_{u_2(W)}^{u_1(W)}\frac{ds}{f(s)^2}=\frac{2\sqrt{2m}}{l\beta}\sqrt{W-W(u_0)}.
\] 
\end{proof}

\begin{lemma}
If in a central field on a surface of revolution $S$ all the orbits near a circular one are closed then the potential $V(u)$ satisfies the differential equation
\beq
\frac{V''(u_0)}{V'(u_0)}=\frac{1}{f'(u_0)f(u_0)}(\beta^2-3(f'(u_0))^2)+\frac{f''(u_0)}{f'(u_0)}.
\label{eqV1}
\eeq
\label{lemmaV1}
\end{lemma}
\begin{proof}
We now Taylor expand the effective potential 
\[
W(u)=V(u)+\frac{l^2}{2mf(u)^2}
\]
around its minimum at $u_0$.
With the notation $W''(u_0)=W_0''$, $u_2(W)=u_0+x$ and $u_1(W)=u_0-y$ we have, to the first non-vanishing order,
\[W-W_0=\frac 1 2 x^2 W_0''+\ldots=\frac 1 2 y^2 W_0''+\ldots.\]
Hence $x=y$ and equation (\ref{eqW}) yields (to this order)
\beq
\left(\int_{u_1}^{u_2}\frac{ds}{f(s)^2}\right)^2=\left(\frac{2x}{f(u_0)^2}\right)^2=\frac{4m}{l^2\beta^2}x^2W''(u_0)
\label{eqnew}\eeq
 The minimum condition
 \[
 W_0'=W'(u_0)=V'(u_0)-\frac{l^2f'(u_0)}{mf^3(u_0)}=0
 \] 
 yields
 \beq
 l^2=\frac{mf^3(u_0)V'(u_0)}{f'(u_0)}.
 \label{eql}
 \eeq
 Substituting Eq. (\ref{eql}) in Eq. (\ref{eqnew}) and using
 \[W_0''=W(u_0)''=V''(u_0)+V'(u_0)\left[-\frac{f''(u_0)}{f'(u_0)}+3\frac{f'(u_0}{f(u_0)}\right]\]
 we obtain (\ref{eqV1}).
\end{proof}

We can now show that the gravitational potential and the harmonic oscillator potential on a surface of revolution $S$ are closed only on some very special surfaces, namely on certain surfaces of constant curvature. 

\begin{proposition}
The gravitational potential $V_1=a\Theta(u)$ gives closed orbits if and only if 
$-f''f+(f')^2=\beta^2$, where $\beta$ is a rational number. The harmonic oscillator potential $V_2=k\Theta(u)^{-2}$ gives closed orbits if and only if $-ff''+(f')^2=\beta^2/4$, where $\beta$ is a rational number.
\end{proposition}
\begin{proof}
Substituting Eq. (\ref{eqV}) in Eq. (\ref{eqV1})  and simplifying we obtain
$-f''f+(f')^2=\beta^2$. The first part of the proof follows from Lemma \ref{lemmakephar}.
Similarly substituting (\ref{eqoscillator}) in Eq. (\ref{eqV1}) and simplifying we obtain
$-f''f+(f')^2=\beta^2/4$. The  proof follows from Lemma \ref{lemmakephar}.

\end{proof}

The following lemma determines the possible values of $\beta$
\begin{lemma}
If in a central field on a surface of revolution $S$ all the orbits near a circular one are closed then we obtain the following equation for $\beta$
\beq
\beta^4-5(-f''f+(f')^2)\beta^2-5f''(f')^2f +4(f'')^2f^2-3f'''f'(f)^2+4(f')^4=0.
\label{eqbeta}
\eeq
\label{lemmabeta}
\end{lemma}
\begin{proof}
We now   Taylor expand the effective potential $V(u)$ around its minimum $u_0$ up to order four
\[
W-W_0=\frac 1 2 x^2 W_0''+\frac 1 6 x^3 W_0'''+\frac {1}{24} x^4 W_0''''+\ldots=
\frac 1 2y^2 W_0''-\frac 1 6 y^3W_0'''+\frac{1}{24} y^4 W_0''''+\ldots, 
\]
and substituting the expansion
$y=x(1+ax+bx^2+\ldots),$ we find $y=x(1+ax+a^2x^2+\ldots)$
with $a=W_0'''/(3W_0'')$.

When this expansion for $y$ is inserted into Eq. (\ref{eqW}) and powers of $x$ up to the fourth order are kept, we obtain,
\[\begin{split}
\left(\int_{u_1}^{u_2}\frac{ds}{f(s)^2}\right)^2&=\frac{x^2}{f(u_0)^4}\left[4 +4ax+\left(5a^2+\frac{8af'(u_0)}{f(u_0)}+\frac{8(f'(u_0)^2}{f(u_0)^2}-\frac 8 3 \frac{f''(u_0)}{f(u_0)}\right)x^2\right]\\
&=\frac{4m}{l^2\beta^2}x^2\left[ W''(u_0)+\frac 1 3 xW'''(u_0)+\frac{1}{12}x^2W''''(u_0)\right].
\end{split}\]
Hence comparing equal powers of $x$
\begin{align}
&\frac{1}{f(u_0)^4}=\left(\frac{m}{l^2\beta^2}\right) W''(u_0)\\
&\frac{a}{f(u_0)^4}=\frac 1 3 \left(\frac{m}{l^2\beta^2}\right) W'''(u_0)\\
&\frac{1}{f(u_0)^4}\left(5a^2+\frac{8af'(u_0)}{f(u_0)}+\frac{8(f'(u_0)^2}{f(u_0)^2}-\frac 8 3 \frac{f''(u_0)}{f(u_0)}\right)=\frac 1 3 
\left(\frac{m}{l^2\beta^2}\right) W''''(u_0)\label{eqW0''''}
\end{align}
The first two equations give Eq. (\ref{eqW}). The new information is contained in the third equation.
Simplifying the expression for the derivatives with the aid of Eqs. (\ref{eql})
and (\ref{eqV1}) we obtain
\begin{align}
&W''(u_0)=\frac{V'(u_0)}{f'(u_0)f(u_0)}\beta^2\label{eqW''}\\
&W'''(u_0)=V'(u_0)\left[\frac{1}{f(u_0)^2}\left(\frac{\beta^2}{(f'(u_0))^2}-7\right)+\frac{f''(u_0)}{(f'(u_0))^2f(u_0)}   \right]\beta^2 \label{eqW'''}\\
&W''''(u_0)=\frac{V'(u_0)}{f'f^3}\left[\frac{\beta^4}{(f')^2}-12\beta^2-\frac{f^2(f'')^2}{(f')^2}-20 f''f +2\frac{f'''f^2}{(f')^2}+47(f')^2\right]\beta^2\label{eqW''''}
\end{align}

thus the quantity $a$ is given by 
\[
a=\frac 1 3 \left [ \frac{f'(u_0)}{f(u_0)}\left(\frac{\beta^2}{(f'(u_0))^2}-7 \right)+\frac{f''(u_0)}{f'(u_0)}\right]
\]
Inserting the last expression and (\ref{eqW''''}) into Eq. (\ref{eqW0''''}) yields (\ref{eqbeta}).
\end{proof}

%\begin{theorem}[Bertrand's Theorem for Surfaces of Constant Curvature]
%In  a central field on a surface of revolution $S$ with constant Gaussian curvature   all the bounded (non rectilinear) orbits are closed if and only if $-ff''+(f')^2=\beta^2$ in which case  the potential energy takes the form $V_1=a\Theta(u)$ or $-ff''+(f')^2=\beta^2/4$  in which case  $V_2=\frac{k}{\Theta^2(u)}$.
We can now prove Bertrand's theorem for surfaces of constant curvature.

\begin{theorem}[Bertrand's Theorem for Surfaces of Constant Curvature]
Consider  an analytic central field on a surface of revolution $S$ with constant Gaussian curvature that has at least one bounded non-circular orbit. Assume the effective potential $W(u)$ has a local minimum. Then all the bounded (non-singular) orbits are closed if and only if $-ff''+(f')^2=\beta^2$ in which case  the potential energy takes the form $V_1=a\Theta(u)$ or $-ff''+(f')^2=\beta^2/4$  in which case  $V_2=\frac{k}{\Theta^2(u)}$.
\label{thbertrand}
\end{theorem}

\begin{proof}
By Lemma \ref{lemmastart} the hypothesis of  Lemma \ref{lemmaV1} and \ref{lemmabeta} are satisfied. 

Since the curvature is constant then $f''=-Kf$ and either $f(u)=Cu+D$ or $f(u)=Ae^{i\sqrt{K}u}+Be^{-i\sqrt{K}u}$.
In the first case from Eq. (\ref{eqbeta}) it follows that 
$\beta^4-5C^2\beta^2+4C^4=0$
and thus either $\beta^2=C^2$ or $\beta^2=4C^2$.
In the second case $\beta^4-20\beta^2KAB+64(KAB)^2=0$
and thus either $\beta^2=4KAB$ or $\beta^2=16KAB$

If $\beta^2=C^2$ or $\beta^2=4KAB$ then by Proposition \ref{prf0} $f(u)$ verifies the equation $-ff''+(f')^2=\beta^2$.
Using Lemma \ref{lemmaV1}, i.e. substituting $-ff''+(f')^2=\beta^2$ into Eq. (\ref{eqV1})
yields
\[
\frac{V''(u)}{V'(u)}=-2\frac{f'(u)}{f(u)}
\]
and solving the previous differential equation we obtain $V=V_1=a\Theta(u)$,
where $\Theta(u)$ is a primitive of $1/f(u)^2$.

On the other hand if $\beta^2=4C^2$ or $\beta^2=16KAB$ then by Proposition \ref{prf0} $f(u)$ verifies the equation $-ff''+(f')^2=\beta^2/4$.
Using Lemma \ref{lemmaV1}, i.e. substituting $-ff''+(f')^2=\beta^2/4$ into Eq. (\ref{eqV1})
yields
\beq
\frac{V''(u)}{V'(u)}=\frac{-3f''(u)f(u)+f'(u)^2}{f(u)f'(u)}.
\label{eqV2}
\eeq
The general solution of the previous equation is of the form $V_2(u)=\frac{k}{\Theta(u)^2}+\mbox{constant}$. To verify it we substitute $V_2$ into Eq. (\ref{eqV2}). We obtain
\[
\frac{6k(f'(u)+f(u)\Theta(u)(-f(u)f''(u)+(f'(u))^2))}{\Theta^4(u)f^4(u)f'(u)}=\frac{6k\left(f'(u)+\frac{\beta^2}{4}f(u)\Theta(u)\right)}{\Theta^4(u)f^4(u)f'(u)}=0
\]
where the last equality follows from Proposition \ref{prf} with $b^2=\beta^2/4$.

To conclude the proof it only remains to check that $V_1$ and $V_2$ do in fact lead to closed orbits. This follows immediately from Lemma \ref{lemmakephar}.
\end{proof}

\begin{remark}
Note that in the statement of Theorem \ref{thbertrand} we added the hypothesis that the central field on the surface $S$ has to have at least one non-circular periodic orbit. This is because there are no bounded orbits near the circular one and therefore the proof breaks down. However there are cases where this  situation arises. For example this condition arises when  one considers the pseudosphere (i.e. a surface of revolution with $f(u)=e^u$) and the gravitational potential  $V_1=a\Theta(u)$.
\end{remark}

We can also show a little more in the case of a general surface of revolution

\begin{theorem}
Consider  an analytic central field on a surface of revolution $S$  that has at least one bounded non-circular orbit. Then there are at most two analytic central potentials on $S$ for which  all the bounded non-singular orbits are closed. There are exactly two (i.e. $V_1=a\Theta(u)$ and  $V_2=\frac{k}{\Theta^2(u)}$) if and only if  $h(u)=-f''f+(f')^2\equiv\mbox{constant}$. There is at most one if $h(u)$ is not identically constant and (\ref{eqbeta}) is verified. In this case the potential  is $V_2=\frac{k}{\Theta^2(u)}$.  
 \end{theorem}
\begin{proof}
By Lemma \ref{lemmastart} the hypothesis of  Lemma \ref{lemmaV1} and \ref{lemmabeta} are satisfied. 

Equation (\ref{eqbeta}) can also be written as
\[
\beta^4-5(-f''f+(f')^2)\beta^2+4(-f''f+(f')^2)^2+3ff'(-f'''f+f'f'')=0.
\]
Substituting  $h(u)=-f''(u)f(u)+(f'(u))^2$ in the previous equation yields
\beq
\beta^4-5h(u)\beta^2+4h(u)^2+3f(u)f'(u)h'(u)=0.
\label{eqh}\eeq

Let $z=\beta^2$ then Eq. (\ref{eqh}) is a quadratic equation in $z$. Let $z_1$ and $z_2$ be the solutions of such equations. Assume $z_1$ and $z_2$ are constant.  Then, since $z_1+z_2=5h(u)$, $h(u)$ must be constant. On the other hand if $h(u)$ is constant $z_1$ and $z_2$ are constant.  This shows that Eq. (\ref{eqh}) has exactly two solution if and only if $h(u)$ is constant. 
From Proposition \ref{prf0} it follows that the surface of revolution $S$ has constant Gaussian curvature. Finally, from  Theorem \ref{thbertrand} it  follows that the two potentials are   $V_1=a\Theta(u)$ and  $V_2=\frac{k}{\Theta^2(u)}$.

Note that equation (\ref{eqV1}) is a first order linear differential equation of the form
\beq
y'(u)+\alpha(u)y=0
\label{eqy}
\eeq
where $y(u)=V(u)$ and $\alpha(u)=\frac{1}{f'f}(\beta^2-3(f')^2)+\frac{f''}{f'}$.
The general solution is of the form $y(u)=Ce^{A(u)}$ where $A'(u)=a(u)$. The expression
$\frac{d}{du}\left(\frac{k}{\Theta^2}\right)=\frac{-2k\Theta'}{\Theta^3}$ (where $\Theta(u)$ is an antiderivative of $1/f^2(u)$)gives the general solution of Eq. (\ref{eqy}) provided  $h(u)$ is not identically equal to $\beta$. In fact let $Ce^{-A(u)}=-2k/(f^2\Theta^3)$ then $A(u)=\ln\left(-\frac{C}{2k}f^2\Theta^3\right)$.
Differentiating $A(u)$, using that $\Theta'(u)=1/f^2(u)$ and simplifying we obtain
\[
A'(u)=\frac{2f'(u)}{f(u)}+\frac{3}{f^2(u)\Theta(u)}=\alpha(u)=\frac{1}{f'(u)f(u)}(\beta^2-3(f'(u))^2)+\frac{f''(u)}{f'(u)}
\]
and solving for $\Theta$ yields
\[
\Theta(u)=\frac{3f'(u)}{f(u)(-\beta^2+(f'(u))^2-f''(u)f(u))}.
\]
Therefore differentiating the expression above, substituting the result in the  equation $\Theta'(u)=1/f^2(u)$ and simplifying we obtain Eq. (\ref{eqbeta}).
Thus, if $f(u)$ satisfies Eq. (\ref{eqbeta}) and $h(u)$ is not identically equal to $\beta$, $y(u)=-\frac{2k\Theta'}{\Theta} $ is a general solution of Eq. (\ref{eqy}) and the corresponding potential is $V_2(u)=\frac{k}{\Theta^2(u)}$.
\end{proof}

%%%%%%%%%%%%%%%%%%%%%%%%%%%%%%%%%%%%%%%%
\section*{\large\bf Acknowledgments}
%%%%%%%%%%%%%%%%%%%%%%%%%%%%%%%%%%%%%%%%%
The author acknowledges with gratitude useful discussions pertinent to the present research with Alain Albouy, Ray McLenaghan and  Cristina Stoica and thanks Ernesto P\'erez-Chavela for bringing to his attention the problem of the motion of a particle on a sphere. 
The research was supported in part by a Wilfried Laurier start-up grant.
%%%%%%%%%%%%%%%%%%%%%%%%%%%%%%%%%%%%%%%%%%%%%%%%%
\section*{\large\bf References}

\parindent 0 pt
\footnotesize
%\begin{thebibliography}{2007}
%%%%%%%%%%%%%%%%%%%%%%%%%%%%%%%%%%%%%%%%%%%%%%%%%
%\bibitem{Appell}

\r{Albouy} A. Albouy, ``Lectures on the two-body problem,"  
in {\it Classical and Celestial Mechanics: The Recife Lectures}, edited by H. Cabral and F. Diacu, (Princeton University Press, Princeton, NJ, 2002).% pp. 63-116

\r{Appell} P. Appell, ``Sur les lois de forces centrales faisant d\'ecrire \'a leur point d'application une conique quelles que soient les conditions initiales," Am. J. Math. {\bf 13}, 153-158 (1891).

\r{Arnold} V.I. Arnol'd, {\it Mathematical Methods of Classical Mechanics}, (Springer-Verlag, New York, 1978)

%\refstepcounter{prova}\label{Borisov}
\r{Borisov} A.V. Borisov, I.S. Mamaev, `` Superintegrable Systems on a Sphere,"  Reg. \& Chaot. Dyn. {\bf 10}, 257-266   (2005).

%\refstepcounter{prova}\label{Carinena}
\r{Carinena} J.F. Cari\~nena, M.F. Ranada  and M. Santander, ``Central Potentials on Spaces of Constant Curvature: The Kepler Problem on the Two-Dimensional sphere $S^2$ and the hyperbolic plane $H^2$," J. Math. Phys. {\bf 46}, 052702-1 (2005). 

%\refstepcounter{prova}\label{Killing}
%\refstepcounter{prova}
\r{Goldstein}
H. Goldstein, {\it Classical Mechanics}, 2nd ed. (Addison-Wesley, Reading, MA, 1980).

\r{Killing} W. Killing,``Die mechanik in den nicht-Euklidischen raumformen," J. Reine Angew. Math. {\bf 98}, 1-48 (1885).

\r{Kozlov}
V.V Kozlov, A.O.  Harin,  ``Kepler's problem in constant curvature spaces"
Cel. Mech Dyn. Astr., {\bf 54}, 393-399 (1992).

\r{Liebmann1902}
H. Liebmann, ``Die Kegelschnitte und die Planetenbewegung im nichteuklidischen Raum" Berichte der K\"oniglich S\"achsischen Gesellschaft der Wissenschaft, Math. Phys. Klasse, {\bf 54}, 393 (1902).

\r{Liebmann}
H. Liebmann,``\"Uber die Zentralbewegung in der Nichteuklidische Geometrie," Leipzig Ber. {\bf 55}, 146-153  (1903).

\r{Lipschitz}
R. Lipschitz, ``Extension of the planet-problem to a space of n dimensions and constant integral curvature," The Quaterly Journal of pure and applied mathematics, {\bf 12}, 349-370 (1873).

\r{Lobachevski}
N.I. Lobachevskij, in {\it Collected Works} (GITTL, Moscow, 1949), Vol. 2, p. 159.

\r{Misner}
C.W. Misner, ``Mixmaster Universe,"  Phys. Rev. Lett. {\bf 22}, 1071-1074 (1969).

\r{Serret}
P. Serret,  {\it Th\'eorie nouvelle g\'eom\'etrique et m\'ecanique des lignes a double courbure}, (Librave de Mallet-Bachelier: Paris, 1860).

\r{Shchepetilov}
A.V. Shchepetilov, ``Comment on ``Central potentials on spaces of constant curvature: The Kepler problem on the two-dimensional sphere S2 and the hyperbolic plane H2" [J. Math. Phys. 46, 052702 (2005)]," J. Math. Phys.{\bf  46}, 114101  (2005).

\r{Schering}
E. Schering, Nachr. K\"onigl. Ges. Wiss. G\"ottingen {\bf 15}, 311 (1870).

\r{Tikochinsky}
Y. Tikochinsky,  ``A simplified proof of Bertrand's theorem," Am. J. Phys.
{\bf 56}, 1073-1075 (1988).

\r{Whittaker}
E.T. Whittaker, {\it A Treatise on the Analytical Dynamics of Particles and Rigid Bodies}, 4th ed. (Cambridge University Press, Cambridge, 1937).

%\end{thebibliography}

\end{document}